\documentclass[12pt,reqno]{amsart}

\usepackage[utf8]{inputenc}
\usepackage[english]{babel}

\usepackage{geometry}
\usepackage{graphicx}
\usepackage{amssymb}
\usepackage{amsmath}
\usepackage{listings}
\usepackage[colorlinks=true,urlcolor=blue,citecolor=blue,linkcolor=blue]{hyperref}

\usepackage{etex}
\usepackage[all]{xy}
\usepackage{tikz}
\usepackage{pgfplots,pgfcalendar}

\theoremstyle{plain}
\newtheorem{theorem}{Theorem}[section] 

\newtheorem{lemma}[theorem]{Lemma} 

\theoremstyle{definition} 
\newtheorem{remark}[theorem]{Remark}

\newtheorem{example}[theorem]{Example}

\numberwithin{equation}{section} 
\newcommand{\doi}[1]{\href{http://dx.doi.org/#1}{doi:{#1}}}

\def\N{\mathbb N}
\def\R{\mathbb R}

\def\C{\mathbb C}

\def\supp{supp}

	\title{NORMING SETS ON A COMPACT COMPLEX MANIFOLD}

\author{Tanaus\'u Aguilar-Hernández}

\date{}

\begin{document}
	
\maketitle

\begin{abstract}
		We describe the norming
	sets for the space of global holomorphic sections to a $k$-power of a positive holomorphic line bundle on a compact complex manifold $X$. We characterize in metric terms the sequence of measurable subsets $\{G_{k}\}_{k}$ of $X$ such that there is a constant $C > 0$ where
	\begin{align*}
	\|s\|^{2}\leq C \int_{G_{k}} |s(z)|^{2}\ dV(z)
	\end{align*}
	for every $s\in H^{0}(L^{k})$ and  for all $k\in\N$.
\end{abstract}

\section{Introduction}

Let $\mathcal{L}$ be a certain space of functions. The problem lies in giving a metric characterization of the set $G$ that verifies that the integral over $G$ is comparable to the norm for every function of $\mathcal{L}$. As we will see below this problem has already proved for some space of functions. 

	In \cite{LogSer} it appears the proof of the Logvinenko-Sereda theorem for  functions of the Paley-Wiener space $PW_{K}$ for a fixed $K$:
	\begin{theorem}
		For a measurable subset $G\subset \R^{n}$ the following are equivalent:
		\begin{enumerate}
			\item[(1)] There is a constant $C>0$ such that
			\begin{align*}
			\int_{\R^{n}} |f|^{2}\ dm\leq C\int_{G} |f|^{2}\ dm
			\end{align*}
			for every $f\in PW_{K}$. We will say that $G$ is a \emph{norming set for the Paley-Wiener space}.
			\item[(2)] There is a cube $Q\subset\R^{n}$ ans a constant $\gamma>0$ such that 
			\begin{align*}
			m((Q+x)\cap G )\geq \gamma
			\end{align*} 		
			for all $x\in\R^n$. We will say that $G$ is \emph{relatively dense in $\R^{n}$}.
		\end{enumerate}
	\end{theorem}
See also \cite[pag 112-115]{VHBJ}.\newline

  The previous theorem has been extended to functions in the Bergman space $A^{p}$ in the unit disk $\mathbb{D}$ in \cite{Lue1}:
\begin{theorem}
	For a measurable subset $G\subset \mathbb{D}$ the following are equivalent:
	\begin{enumerate}
		\item[(1)] There is a constant $C>0$ such that
		\begin{align*}
		\int_{\mathbb{D}} |f|^{p}\ dm \leq C\int_{G} |f|^{p}\ dm
		\end{align*}
		for every $f\in A^{p}$. We will say that $G$ is a \emph{norming set for the Bergman space}.
		\item[(2)] There is a constant $\delta>0$ and radius $R\in (0,1)$ such that
		\begin{align*}
		m(G\cap D(a,R))>\delta m(D(a,R))
		\end{align*}
		for all $a\in\mathbb{D}$. We will say that $G$ is \emph{relatively dense in the disc $\mathbb{D}$}.
	\end{enumerate}
\end{theorem}

\begin{remark}
	In \cite{Lue1} a different geometry than usual will be used. The disks $D(a,R)$ will be of the form
	\begin{align*}
	D(a,R)=\{z\in\mathbb{D}\ :\ |z-a|<R(1-|a|)\},
	\end{align*}
	where $a\in\mathbb{D}$ and $R\in (0,1)$.
\end{remark} 

It should be noted that using similar arguments to those used in \cite{Lue1}  one can prove the Logvinenko-Sereda theorem for the classical Fock space $\mathcal{F}^{2}_{2|z|^{2}}(\C^{n})$:
\begin{theorem}
	For a a measurable subset $G\subset \C^{n}$ the following are equivalent:
	\begin{enumerate}
		\item[(1)] There is a constant $C>0$ such that
		\begin{align*}
		\int_{\C^{n}} |f|^{2} e^{-2|z|^{2}}\ dm(z)\leq C\int_{G} |f|^{2} e^{-2|z|^{2}}\ dm(z)
		\end{align*}
		for every $f\in \mathcal{F}^{2}_{2|z|^{2}}(\C^{n})$. We will say that $G$ is a \emph{norming set for the classical Fock space.}
		\item[(2)] There is a constant $\delta>0$ and a radius $R>0$ such that 
		\begin{align*}
		m(G\cap G(z,R))\geq \delta m(B(a,R))
		\end{align*}
		for all $z\in\C^{n}$. We will say that $G$ is \emph{relatively dense in $\C^{n}.$}
	\end{enumerate}
\end{theorem}

Our main goal will be to prove the Logvinenko-Sereda theorem for the space of global holomorphic sections to a $k$-power of a positive line bundle on a compact complex manifold $X$. Before giving the statement of the theorem, we will present the convenient definitions.

First of all, we  consider a compact complex manifold $X$ which will be endowed with a smooth Hermitian metric $\omega$. As we know, this metric induces a distance function $d(x,y)$ on $X$, which will be used to define the balls, and a volume form $V$, which will be used to integrate over the manifold.

We assume that the holomorphic line bundle $L$ on a compact complex manifold $X$ is endowed with a smooth hermitian metric $\phi$. 

We will denote by $H^{0}(L)$ the space of global holomorphic sections to $L$. Moreover, we will consider that the line bundle $(L,\phi)$ is positive. 

As $\phi$ is a Hermitian metric on $L$, then $\partial\overline{\partial}\phi$ is a globally defined (1,1)-form on $X$, which is called the curvature form of the metric $\phi$. Moreover, the line bundle $L$ with the metric $\phi$ is called positive if $i\partial\overline{\partial}\phi$  is a positive form. 
The statement of the Logvinenko-Sereda theorem is the following:

\begin{theorem}For a sequence of measurable subsets $\{G_{k}\}$ in $X$ the following are equivalent:
	\begin{enumerate}
		\item[(1)] There is a constant $C>0$ such that for all $k\in\N$,
		\begin{equation*}
		\int_{X} |s(x)|^{2}\ dV(x)\leq C\int_{G_{k}} |s(x)|^{2}\ dV(x)
		\end{equation*}
		for every $s\in H^{0}(L^{k})$.  We will say that $G_k$ is a \emph{norming set in $H^{0}(L^{k})$.}
		\item[(2)] There is a constant $\delta>0$ and a radius $R$ such that 
		\begin{align*}
		V\left(G_{k}\cap B\left(a,\frac{R}{\sqrt{k}}\right)\right)>\delta V\left(B\left(a,\frac{R}{\sqrt{k}}\right)\right)
		\end{align*}
		for all $a\in X$ and $k\in\N$. We will say that $G_k$ is \emph{relatively dense in $X$}.
	\end{enumerate}
\end{theorem}

\begin{remark}[Bergman kernel]
	The space $H^{0}(L)$ admits a Hilbert space structure endowed with the scalar product
	\begin{equation*}
	\langle u,v\rangle=\int_{X} \langle u(x),v(x)\rangle,\quad u,v\in H^{0}(L)
	\end{equation*}
	where the integration is taken with respect to the volume form $V$.
	
	The Bergman kernel $\Pi(x,y)$ associated to this space is a section to the line bundle $L\boxtimes \overline{L}$ over the manifold $X\times X$, defined by 
	\begin{align*}
	\Pi(x,y)=\sum_{j=1}^{N} s_{j}(x)\otimes \overline{s_{j}(y)}
	\end{align*}
	where $s_{1}, \dots, s_{N}$ is an orthonormal basis for $H^{0}(L)$. Moreover, this definition does not depend on the choice of te orthonormal basis. 
	Notice that $L\boxtimes \overline{L}=\pi^{\ast}_{1}(L)\otimes\pi^{\ast}_{2}(\overline{L})$ where $\pi_{i}=X\times X\rightarrow X$ is the projection onto the $i$-factor.
	
	The Bergman kernel $\Pi(x,y)$ is in a sense the reproducing kernel for the space $H^{0}(L)$, satisfying the reproducing formula 
	\begin{align*}
	s(x)=\int_{X}\langle s(y), \Pi(x,y)\rangle\ dV(y)
	\end{align*}
	for $s(y)\in H^{0}(L)$. 
	
	The pointwise norm of the Bergman kernel in symmetric, $|\Pi(x,y)|=|\Pi(y,x)|$. Moreover, it satisfies 
	\begin{align*}
	|\Pi(x,x)|=\int_{X} |\Pi(x,y)|^{2}\ dV(y).
	\end{align*}
	
	\begin{lemma}\label{lemm}
		Let $(L,\phi)$ be a positive line bundle. We have the off-diagonal estimate
		\begin{align*}
		|\Pi_{k}(x,y)|\lesssim k^{n} e^{-c\sqrt{k}\ d(x,y)}
		\end{align*}
		where $c$ is an appropriate positive constant and the diagonal estimate
		\begin{align*}
		|\Pi_{k}(x,x)|\asymp k^{n}.
		\end{align*}
		Moreover, from this we obtain the estimate of the dimension of $H^{0}(L^k)$
		\begin{align*}
		\dim H^{0}(L^{k})\asymp k^{n}.
		\end{align*}
	\end{lemma}

		See \cite{BoB} where this lemma is proved using a method that appears in \cite{LIND}.

\end{remark}

\section{Previous result}

The main goal of the lemma is to prove for each ball the existence of a normalized peak-section, that is, a section such that most of the mass of the function is in such ball. For this we will use the reproducing kernel of $H^{0}(L^{k})$.

\begin{lemma}\label{L1}
	Given $\varepsilon>0$, there is a radius $R>0$ such that for all ball $B\left(\xi,\frac{R}{\sqrt{k}}\right)$, there is a section $s=s_{\xi,k}\in H^{0}(L^{k})$ such that
	\begin{itemize}
		\item $\|s\|^{2}={\int_{X} |s(x)|^{2}\ dV(x)=1}$,\newline
		
		\item ${\int_{X\setminus B\left(\xi,\frac{R}{\sqrt{k}}\right)} |s(x)|^{2}\ dV(x)<\varepsilon}$,\newline
	\end{itemize} 
\end{lemma}

\begin{proof}
	Given $\varepsilon>0$. We consider the reproducing kernel of the Hilbert space $H^{0}(L^k)$
	\begin{align*}
	\Pi_{k}(x,y)=\sum_{j=1}^{N} s_{j}(x)\otimes \overline{s_{j}(y)}
	\end{align*}
	where $s_{1},\dots, s_{N}$ is an orthonormal basis for $H^{0}(L^k)$.
	
	If we fix $y\in X$, there is a section $\Phi_{k,y}\in H^{0}(L^k)$ such that 
	$$|\Phi_{k,y}(x)|=|\Pi_{k}(x,y)|,\quad x\in X,$$
	by the lemma in \cite[pag. 431]{LevJoa}.
	
	We will consider the section 
	\begin{align*}
	g_{k,y}(x)=\frac{\Phi_{k,y}(x)}{\sqrt{|\Pi_{k}(y,y)|}},\quad x\in X,
	\end{align*}
	where $\Pi_{k}$ denotes the Bergman kernel for the $k$'th power $L^{k}$ of the line bundle $L$.

	Let us see the first property.
	\begin{align*}
	\int_{X} |g_{k,y}(x)|^2\ dV(x)=\frac{1}{|\Pi_{k}(y,y)|} \int_{X} |\Pi_{k}(x,y)|^{2}\ dV(x)=\frac{|\Pi_{k}(y,y)|}{|\Pi_{k}(y,y)|}=1.
	\end{align*}

Now, we check the second property
\begin{align*}
&\int_{X\setminus B\left(y,\frac{R}{\sqrt{k}}\right)} |g_{k,y}(x)|^{2}\ dV(x)=\frac{1}{|\Pi_{k}(y,y)|}\int_{X\setminus B\left(y,\frac{R}{\sqrt{k}}\right)} |\Pi_{k}(x,y)|^{2}\ dV(x)\\
&= \frac{1}{|\Pi_{k}(y,y)|} \int_{0}^{\infty} V\left(\left\{x: |\Pi_{k}(x,y)|>\lambda  \right\}\setminus B\left(y,\frac{R}{\sqrt{k}}\right)\right)2\lambda\ d\lambda\\
&\leq \frac{1}{|\Pi_{k}(y,y)|}\int_{0}^{k^n D}  V\left(\left\{x: De^{-c\sqrt{k}\ d(x,y)}\geq k^{-n}\lambda  \right\}\setminus B\left(y,\frac{R}{\sqrt{k}}\right)\right)2\lambda\ d\lambda.
\end{align*}
where we use that $|\Pi_{k}(x,y)|\leq k^{n} D e^{-c\sqrt{k}\ d(x,y)}$ for some constants $c,D>0$, as we see in Lemma~\ref{lemm}. 
Applying the change of variable $\lambda=k^{n} D e^{-c u}$.
\begin{align*}
&\frac{1}{|\Pi_{k}(y,y)|}\int_ {0}^{\infty} V\left(B\left(y,\frac{u}{\sqrt{k}}\right)\setminus B\left(y,\frac{R}{\sqrt{k}}\right)\right) (2k^{n}De^{-cu}) |-ck^{n} De^{-cu}|\ du\\
&\lesssim \frac{1}{k^{n}} \int_ {R}^{\infty}\left(\frac{u}{\sqrt{k}}\right)^{2n} (2k^{n}De^{-cu}) |-ck^{n} De^{-cu}|\ du\\
&\lesssim\int_{R}^{\infty} u^{2n} e^{-2cu}\ du = e^{-2cR}\int_{0}^{\infty} (\mu+R)^{2n} e^{-2c\mu}\ d\mu\\
& = e^{-2cR} \sum_{j=0}^{2n} \binom{2n}{k} R^{2n-j} \int_{0}^{\infty} \mu^{j} e^{-2c\mu}\ d\mu\\
&\lesssim e^{-2cR} \sum_{j=0}^{2n} \binom{2n}{k}\frac{R^{2n-j}}{(2c)^{j}}  j!\lesssim e^{-2cR} \sum_{j=0}^{2n} \binom{2n}{k}(2cR)^{2n-j}=e^{-2cR}(1+2cR)^{2n}.
\end{align*}

Finally, for $R$ large enough we obtain that 
\begin{align*}
\int_{X\setminus B\left(y,\frac{R}{\sqrt{k}}\right)} |g_{k,y}(x)|^{2}\ dV(x)&<\varepsilon.
\end{align*}

\end{proof}

\begin{remark}\label{remark22}
Notice that taking the same section of the previous lemma and	using the Cauchy-Schwarz inequality we obtain that
	\begin{align*}
	|s(x)|^{2}=|\langle s(y),\Pi_{k}(x,y)\rangle|^{2}\leq \|s\|^{2} |\Pi(x,x)|\lesssim k^{n}.
	\end{align*}
\end{remark}

\section{Main results}

\begin{theorem} \label{TeoLogSe}
	For a sequence of measurable subset $\{G_{k}\}_{k}$ in $X$ the following are equivalent:
	\begin{enumerate}
		\item[(1)] There is a constant $C>0$ such that for all $k\in\N$,
		\begin{align*}
		\int_{X}\ |s(x)|^{2}\ dV(x) \leq C \int_{G_{k}} |s(x)|^{2}\ dV(x)
		\end{align*}
		for every $s\in H^{0}(L^k)$. We will say that $G_k$ is a \emph{norming set of $H^{0}(L^{k})$}.
		\item[(2)] There is a constant $\delta>0$ and a radius $R$ such that 
\begin{align*}
		V\left(G_{k}\cap B\left(a,\frac{R}{\sqrt{k}}\right)\right)>\delta V\left(B\left(a,\frac{R}{\sqrt{k}}\right)\right)
\end{align*}
for all $a\in X$ and $k\in\N$. We will say that $G_k$ is \emph{relatively dense in X}.
	\end{enumerate}
\end{theorem}
\begin{proof}
	The proof that (1) implies (2) is the easiest. In this proof, we consider a section
	with the properties of Lemma~\ref{L1}.
	
	So, given $\varepsilon \leq 1/2C$ and applying the Lemma~\ref{L1} there is a radius $R$ such that for all balls $B\left(\xi, \frac{R}{\sqrt{k}}\right)$ there is a section $s$ verifying the properties of the lemma. Hence, applying the Remark~\ref{remark22} we obtain
	\begin{align*}
	\frac{V\left(G_{k}\cap B \left(\xi, \frac{R}{\sqrt{k}}\right)\right)}{V\left(B \left(\xi, \frac{R}{\sqrt{k}}\right)\right)}&\gtrsim \frac{1}{R^{2n}} \int_{G_{k}\cap B\left(\xi, \frac{R}{\sqrt{k}}\right)} |s^{\ast}(x)|^{2}\ dV(x)
	\end{align*}
	and using the other properties with (2) we have
	\begin{align*}
	\frac{V\left(G_{k}\cap B \left(\xi, \frac{R}{\sqrt{k}}\right)\right)}{V\left(B \left(\xi, \frac{R}{\sqrt{k}}\right)\right)}&\gtrsim\frac{1}{R^{2n}} \left(\int_{G} |s^{\ast}(x)|^{2}\ dV(x)-\int_{X\setminus B \left(\xi, \frac{R}{\sqrt{k}}\right) } |s^{\ast}(x)|^{2}\ dV(x) \right)\\
	&\geq \frac{1}{R^{2n}}\left(\frac{1}{C}-\varepsilon\right)\geq \frac{1}{2CR^{2n}}.
	\end{align*}
	
	Therefore, we have proved that (1) implies (2).\newline
	
The proof that (2) implies (1) is the difficult one. We will use a similar argument as in \cite[Chapter 4]{JM}. We will partition $X$ in two pieces. The first piece, denoted by $\mathcal{A}$, are the points where the sections is much smaller that its average and the second is the complementary $X\setminus\mathcal{A}$. The following lemma proves that the integral over $\mathcal{A}$ is irrelevant as most of the mass is carried by $X\setminus \mathcal{A}$.

\begin{lemma}\label{l2}
	Let $\varepsilon>0$ and $s\in H^{0}(L^{k})$. Define the set
	\begin{align*}
	\mathcal{A}=\left\{a\in X\ |\ |s(a)|^{2}<\frac{\varepsilon}{V(B(a,R/\sqrt{k}))}\int_{B(a,R/\sqrt{k})}|s(z)|^{2}\ dV(z)\right\}.
	\end{align*}
	Then there is a constant $C$ depending on $R$ and $k$ such that 
	\begin{align*}
	\int_{\mathcal{A}} |s(z)|^{2}\ dV(z)<C\varepsilon \int_{X} |s(z)|^{2}\ dV(z).
	\end{align*}
\end{lemma}

\begin{proof}
	
	For $\xi\in \mathcal{A}$ we have 
	\begin{align*}
	|s(a)|^{2}<\frac{\varepsilon}{V(B(a,R/\sqrt{k}))} \int_{B(a,R/\sqrt{k})}|s(z)|^{2}\ dV(z) =\varepsilon \int_{X} |s(z)|^{2}\frac{\chi_{B(a,R/\sqrt{k})}(z)}{V(B(a,R/\sqrt{k}))}\ dV(z).
	\end{align*}
	
	Integrating respect to $a$ and applying the Fubini's theorem
	\begin{align*}
	\int_{\mathcal{A}} |s(a)|^{2}\ dV(a)&<\varepsilon \int_{X}|s(z)|^{2}\left(\int_{\mathcal{A}}\frac{\chi_{B(a,R/\sqrt{k})}(z)}{V(B(a,R/\sqrt{k}))}\ dV(a) \right)\ dV(z)\\
	&\leq \frac{V(B(a,R/\sqrt{k}))}{V(B(a,R/\sqrt{k}))}\varepsilon \int_{X} |s(z)|^{2}\ dV(z)=\varepsilon \int_{X} |s(z)|^{2}\ dV(z).
	\end{align*}
	
	Therefore, we obtain
	\begin{align*}
\int_{\mathcal{A}} |s(a)|^{2}\ dV(a)\leq \varepsilon \int_{X} |s(z)|^{2}\ dV(z).
	\end{align*}
\end{proof}

Let $\mathcal{F}=X\setminus\mathcal{A}=\left\{a\in X\ :\ |s(a)|^{2}\geq \frac{\varepsilon}{V(B(a,R/\sqrt{k}))} \int_{B(a,R/\sqrt{k})} |s(z)|^{2}\ dV(z)\right\}$.
If we choose $\varepsilon$ such that $\varepsilon C<1/2$ we have 
\begin{align*}
\int_{X} |s(z)|^{2}\ dV(z)<2\int_{\mathcal{F}} |s(z)|^{2}\ dV(z)
\end{align*}
since 
\begin{align*}
\int_{X} |s(z)|^{2}\ dV(z)&=\int_{\mathcal{F}} |s(z)|^{2}\ dV(z)+\int_{\mathcal{A}} |s(z)|^{2}\ dV(z)\\
&<\int_{\mathcal{F}} |s(z)|^{2}\ dV(z)+C\varepsilon\int_{X} |s(z)|^{2}\ dV(z)\\
&<\int_{\mathcal{F}} |s(z)|^{2}\ dV(z)+\frac{1}{2}\int_{X} |s(z)|^{2}\ dV(z).
\end{align*}

Thus this is enough to show that
\begin{align*}
\int_{\mathcal{F}} |s(z)|^{2}\ dV(z)\lesssim \int_{G_{k}} |s(z)|^{2}\ dV(z).
\end{align*}

All we need to prove is the existence of a constant $C>0$ such that for all $w\in\mathcal{F}$
\begin{align}\label{ineq1}
|s(w)|^{2}\leq \frac{C}{V\left(B\left(w,\frac{R}{\sqrt{k}}\right)\right)} \int_{B\left(w,\frac{R}{\sqrt{k}}\right)\cap G_{k}} |s(z)|^{2}\ dV(z).
\end{align}

Indeed, if this is the case then
\begin{align*}
\int_{\mathcal{F}} |s(w)|^{2}\ dV(w)&\leq C\int_{G_{k}} |s(z)|^{2}\left(\int_{\mathcal{F}}\frac{\chi_{B\left(w,\frac{R}{\sqrt{k}}\right)}(z)}{V\left(B\left(w,\frac{R}{\sqrt{k}}\right)\right)}\ dV(w)\right)\ dV(z)\\
&\leq C\int_{G_{k}} |s(z)|^{2}\ dV(z).
\end{align*}

To prove \eqref{ineq1} we argue by contradiction. If \eqref{ineq1} is false there are for any $q\in\N$ sections $s_{q}\in H^{0}(L^{k_{q}})$ and $w_{q}\in\mathcal{F}$ such that
\begin{align}\label{ineq2}
|s_{q}(w_{q})|^{2}>\frac{q}{V\left(B\left(w_{q},\frac{R}{\sqrt{k_{q}}}\right)\right)}\int_{B(w_{q},\frac{R}{\sqrt{k_{q}}})\cap G_{k_{q}}} |s_{q}(z)|^{2}\ dV(z)
\end{align}

By compactness of $X$ we can choose a convergent subsequence $\{w_{q}^{\ast}\}$ of $w_{q}$ to some $w^{\ast}\in X$ such that there is a local chart $(U,\varphi)$ where 
\begin{align*}
B\left(w_{q}^{\ast},\frac{R}{\sqrt{k_{q}}}\right)\subset U,\ \forall q\in\N
\end{align*}
and $\varphi(w^{\ast})=0$.

By the positivity of $i\partial\overline{\partial}\phi$, we can make a linear change of variable so that $i\partial\overline{\partial}\phi(\varphi^{-1}(0))=i\partial\overline{\partial}|z|^{2}$. So if we consider the Taylor series of $\phi$ in the local ball $B(0,\tau)$ we obtain 
\begin{align*}
\phi(\varphi^{-1}(z))=h(z)+|z|^{2}+q(z)
\end{align*}
where $h$ is a pluriharmonic polynomial of degree less than or equal to 2, ${h(0)=\phi(\varphi^{-1}(0))}$ and $q(z)=o(|z|^{2})$. Hence, we have that
\begin{align*}
|2\phi(\varphi^{-1}(z))-2|z|^{2}-2h(z)|\leq c(\tau)\tau^{2}
\end{align*}
in $B(0,\tau)$ for nondecreasing continuous function $c(\tau)$ such that $c(0)=0$. Notice that we can use the local chart $(\varphi^{-1}(B(0,\tau{\ast})),\varphi)$ instead since $\exists\nu\in\N$, $\forall q\geq \nu$ : $B\left(w_ {q}^{\ast},\frac{R}{\sqrt{k_{q}}}\right)\subset\varphi^{-1}(B(0,\tau^{\ast}))$ and $\tau^{\ast}\leq \tau$.

Now we consider the sequence of holomorphic functions 
\begin{align*}
G_{q}(z)=f_{q}(\varphi^{-1}(z))e^{-\frac{k}{2}H(z)},\quad q\in\N
\end{align*}
where $\Re(H)=h$ and $|s_{q}|^{2}=|f_{q}|^{2}e^{-k\phi}$.

Next we need obtain that 
$$|s_{q}(w_{q}^{\ast})|^{2}\lesssim |G_{q}(\eta_{q})|^{2}$$ 
where $\varphi(w_{q}^{\ast})=\eta_{q}$, $$|s_{q}(\varphi^{-1}(z))|^{2}\gtrsim G_{q}(z)|^{2}$$ 
for a certain ball and 
\begin{align}\label{ineq3}
|G_{q}(\lambda_{q})|^{2}\gtrsim \frac{q}{m(B(\eta_{q},\delta_{q}))} \int_{B(\eta_{q},\delta_{q})\cap \varphi(E_{k_{q}})} |G_{q}(z)|^{2}\ dm(z).
\end{align}

First, for the sequence $\{w_{q}^{\ast}\}$ we have
\begin{align*}
|s_{q}(w_{q}^{\ast})|^{2}=|f_{q}(\varphi^{-1}(\eta_{q}))|^{2} e^{-k\phi(\varphi^{-1}(\eta_{q}))}\leq |f_{q}(\varphi^{-1}(\eta_{q}))|^{2}e^{-kh(\eta_{q})} e^{\frac{k}{2}c(\tau^{\ast})\tau^{\ast}}\lesssim |G_{q}(\eta_{q})|^{2}. 
\end{align*}
Furthermore for $z\in B(0,\tau^{\ast})$ we obtain
\begin{align*}
|G_{q}(z)|^{2}=|f_{q}(\varphi^{-1}(z))|^{2} e^{-kh(z)}\lesssim |s_{q}(\varphi^{-1}(z))|^{2},\quad z\in B(0,\tau^{\ast}).
\end{align*}

Finally we obtain the inequality that we need
\begin{align*}
|G_{q}(\eta_{q})|^{2}&\gtrsim |s_{q}(w_{q}^{\ast})|^{2}\gtrsim \frac{q}{V\left(B\left(w_{q}^{\ast},\frac{R}{\sqrt{k_{q}}}\right)\right)}\int_{B\left(w_{q}^{\ast},\frac{R}{\sqrt{k_{q}}}\right)\cap E_{k_q}}|s_{q}(z)|^{2}\ dV(z)\\
&\gtrsim \frac{q}{V\left(B\left(w_{q}^{\ast},\frac{R}{\sqrt{k_{q}}}\right)\right)}\int_{\varphi(B\left(w_{q}^{\ast},\frac{R}{\sqrt{k_{q}}}\right)\cap E_{k_q})} |f_{q}(\varphi^{-1}(z))|^{2}\ e^{-k\phi(\varphi^{-1}(z))}\ dm(z)\\
&\gtrsim \frac{q}{m(B(\eta_{q},\delta_{k_{q}}))}\int_{B\left(\eta_{q},\delta_{k_{q}}\right)\cap E_{k_{q}}^{\ast}} |G_{q}(z)|^{2}\ dm(z)
\end{align*}
where $E_{k_{q}}^{\ast}=\varphi(E_{k_{q}})$ and there is a ball $B(\eta_{k_{q}},\delta_{k_{q}})\subset \varphi(w_{q}^{\ast},R/\sqrt{k_{q}})$ such that $m(B(\eta_{q},\delta_{k_{q}}))\asymp V\left(B\left(w_{q}^{\ast},\frac{R}{\sqrt{k_ {q}}}\right)\right)$, because as $\varphi,\varphi^{-1}\in\mathcal{C}^{1}$ we have that $\varphi$ is bi-Lipschitz.

By means of a dilatation and a translation we send $\eta_{q}$ to the origin of $\C^{n}$, the ball $B(\eta_{q}, \delta_{k_{n}})$ to $B(0,1)\subset \C^{n}$ and the set $E_ {k_ {q}}^{\ast}$ to $E_{q}^{\ast}$. We will denote the set $\overline{B(0,1)}\cap E_{q}^{\ast}$ by $H_{q}$.

Moreover, multiplying by a constant we can assume that
\begin{align*}
\int_{B(0,1)} |\tilde{G}_{q}(z)|^{2}\ dm=1.
\end{align*}

The subharmoniticity of $|\tilde{G}_{q}|^{2}$ and the fact that $w_ {q}^{\ast}\in \mathcal{F}$ tells us that 
\begin{align*}
\varepsilon\lesssim |\tilde{G}_{q}(0)|^{2}\lesssim 1.
\end{align*}
Notice that we can modify the inequality of the definition of $\mathcal{F}$ in the same way as the expression \eqref{ineq3}.

And this property together with \eqref{ineq3} yields
\begin{align}\label{ineq4}
\frac{1}{q}\gtrsim \int_{H_{q}} |\tilde{G}_{q}(u)|^{2}\ dm(u).
\end{align}
 We have that $\{\tilde{G}_{q}\}$ is a locally bounded sequence of holomorphic functions defined in $B(0,1)$ and therefore using Montel theorem there exist a subsequence converging locally uniformly on $B(0,1)$ to some holomorphic function $G$.

We observe that the relative dense hypothesis yields 
\begin{align*}
\inf_{n} m(H_{n})>0.
\end{align*}

The Helly selection theorem guaranties the existence of a  weak-$\ast$limit $\tau$ of a subsequence of $\tau_{q}=\chi_{H_{q}} m$ such that $\tau\not\equiv 0$. Condition \eqref{ineq4} implies that $\tilde{G}=0$ $\tau$-a.e. and therefore $\supp\tau\subset\{\tilde{G}=0\}$.

We want to show that $\supp\tau$ cannot lie on a complex $(n-1)$-dimensional submanifold $S\subset \C^n$. We argue by contradiction.

By definition of $\tau_{n}$ we have $\tau_{n}(B(x,r))\leq m(B(x,r))\lesssim r^{2n}$. As this inequality is true for all $\tau_{q}$, the same holds for $\tau$. But, using the Frostman lemma we obtain that $\dim_{Haus}(\supp \tau)\geq 2n>0$.

Therefore, we have proved that (2) implies (1).\newline
\end{proof}

Now, we will generalize the previous sampling problem using a sequence of measures $\{\mu_k\}_{k}$ of $\mathfrak{M}(X)$ instead of the particular sequence of measures $\{\chi_{G_k} V\}_{k}$ of $\mathfrak{M}(X)$, where $\{G_k\}_{k}$ are measurable sets. We will solve part of this problem, the remaining part continues being an open problem.

\begin{theorem}
	For a sequence of measures $\{\mu_{k}\}$ the following are equivalent:
	\begin{enumerate}
		\item[(1)] There exists a constant $C_{1}>0$ such that for all $k\in\N$,
		\begin{align*}
		\int_{X} |s(z)|^{2}\ d\mu_{k}(z)\leq C_{1} \int_{X} |s(z)|^{2}\ dV(z)
		\end{align*}
		for every $s\in H^{0}(L^{k})$.
		\item[(2)] There exists a constant $C_{2}>0$ such that for all $k\in\N$,
		\begin{align*}
		T[\mu_k](z)=\int_{X} \left|\frac{\Pi_{k}(w,z)}{\sqrt{|\Pi_{k}(z,z)}|}\right|^{2}d\mu_{k}(w)\leq C_{2} 
		\end{align*}
		for every $z\in X$, that is, the sequence of Berezin transforms $\{T[\mu_k]\}_{k}$ are uniformly bounded in $X$.
		\item[(3)] There is a constant $C_{3}>0$ such that 
		\begin{align*}
		\mu_{k}(B(z,1/\sqrt{k}))\leq \frac{C_{2}}{k^{n}},\qquad z\in X.
		\end{align*}
	\end{enumerate}
\end{theorem}
\begin{proof}
	The proof that (1) implies (2) is trivial because $\Pi_{k}(z,w) \in H^{0}(L^ {k})$ for each $w\in X$ and $k\in\N$.\newline
	 
	 To prove that (2) implies (3) is enough the following lemma.
	 
	 \begin{lemma}\label{lem34}
	 There are constants $\varepsilon>0$, $M>0$ and $k_{0}\in\N$ such that
	 \begin{align}\label{ineqcar1}
	 \left|\frac{\Pi_{k}(w,z)}{\sqrt{|\Pi_{k}(z,z)|}}\right|^{2}\geq M k^{n}>0
	 \end{align} 
	 for every $w\in B(z,\varepsilon/\sqrt{k})$ and $k>k_{0}$.	
	 \end{lemma} 
 \begin{proof}
 	 To prove \eqref{ineqcar1} we use that 
 \begin{align*}
 \left|\frac{\Pi_{k}(z,z)}{\sqrt{|\Pi_{k}(z,z)|}}\right|^{2}-\left|\frac{\Pi_{k}(w,z)}{\sqrt{|\Pi_{k}(z,z)|}}\right|^{2}=\int_{0}^{1} \frac{d}{dt}\left(\left|\frac{\Pi_{k}(w-(z-w)t,z)}{\sqrt{|\Pi_{k}(z,z)|}}\right|^{2}\right)dt 
 \end{align*}
 and by the bound of the derivatives in \cite[pag. 434]{LevJoa} 
 \begin{align*}
 \left|\frac{\partial}{\partial \xi_{j}}\left|\frac{\Pi_{k}(w,z)}{\sqrt{|\Pi_{k}(z,z)|}}\right|^{2}\right|  \leq C \sqrt{k} \sup_{\gamma\in U_{k}(w)}   \left|\frac{\Pi_{k}(\gamma,z)}{\sqrt{|\Pi_{k}(z,z)|}}\right|^{2}
 \end{align*}
 where
 \begin{align*}
 U_{k}(w)=\{x\in X\:\ |z_{j}-w_{j}|<1/\sqrt{k},(j=1,\dots,n)\},
 \end{align*}
 together with the sub-mean value property in \cite[pag. 432 ]{LevJoa}, we obtain that
 \begin{align*}
 \left|\frac{\Pi_{k}(z,z)}{\sqrt{|\Pi_{k}(z,z)|}}\right|^{2}-\left|\frac{\Pi_{k}(w,z)}{\sqrt{|\Pi_{k}(z,z)|}}\right|^{2}\leq M'k^{\frac{2n+1}{2}} \sqrt{n}  |z-w|.
 \end{align*}
 
 By the diagonal estimate of Lemma~\ref{lemm} we have
 \begin{align*}
 \left|\frac{\Pi_{k}(w,z)}{\sqrt{|\Pi_{k}(z,z)|}}\right|^{2}\geq D k^{n}-M'k^{\frac{2n+1}{2}}\sqrt{n} |z-w|,
 \end{align*}
 where $D>0$ is a certain constant.
 
 Then, if we consider the ball $B(z,\varepsilon/\sqrt{k})$ where $\varepsilon<\frac{D}{2M'\sqrt{n}}$ we have
 \begin{align*}
 \left|\frac{\Pi_{k}(w,z)}{\sqrt{|\Pi_{k}(z,z)|}}\right|^{2}\geq (D-M'\varepsilon\sqrt{n})k^{n}>\frac{Dk^{n}}{2}>0
 \end{align*}
 for $w\in B(z,\varepsilon/\sqrt{k})$.
 \end{proof}
 
 Applying  Lemma~\ref{lem34} we have that there are constants $\varepsilon>0$, $M>0$ and $k_{0}\in\N$ such that
 \begin{align*}
 M k^{n} \mu_{k}(B(z,\varepsilon/\sqrt{k}))\leq \int_{B(z,\varepsilon/\sqrt{k})} \left|\frac{\Pi_{k}(w,z)}{\sqrt{|\Pi_{k}(z,z)|}}\right|^{2}\ d\mu_{k}(w) \leq T[\mu_k](z)\leq C_{2}
 \end{align*}
 for every $z\in X$ and $k>k_{0}$. 
	
	We take a covering of the ball $B(z,1/\sqrt{k})$ by a collection of balls $B(z_{j},\varepsilon/5\sqrt{k})$, $z_{j}\in B(z,1/\sqrt{k})$ and applying the Vitali covering lemma we obtain a sub-collection of these balls which are disjoint such that
	\begin{align*}
	\cup B(z_{j},\varepsilon/5\sqrt{k})\subset \cup B(z_{j_{q}},\varepsilon/\sqrt{k}).
	\end{align*}
	Moreover, as 
	\begin{align*}
	B(z_{j_{q}},\varepsilon/\sqrt{k})\subset B(z,2/\sqrt{k})
	\end{align*}
	by means of the volume we have that $$N\cdot \left(\frac{\varepsilon}{\sqrt{k}}\right)^{2n}\lesssim \left(\frac{2}{\sqrt{k}}\right)^{2n},$$ that is, $	N\lesssim \left(\frac{2}{\varepsilon}\right)^{2n}$ where $N$ is the number of balls of the sub-collection. Therefore, we conclude 
	\begin{align*}
	\mu_{k}(B(z,1/\sqrt{k}))\lesssim\left(\frac{2}{\varepsilon}\right)^{2n} m_{k}(B(z_{j_{q}}),\varepsilon/\sqrt{k}) \lesssim \frac{1}{k^{n}}. 
	\end{align*}

	 Notice that for $k\leq k_{0}$ (3) holds immediately since $\mu_{k}(B(z,1/\sqrt{k}))$ are bounded by $\mu_{k}(X)$.\newline

	 All we need to prove that (3) implies (1) is the existence of a constant $Q>0$ such that for all $w\in X$
	\begin{align}\label{eqafinal}
	|s(w)|^{2}\leq Q k^{n} \int_{B(w,1/\sqrt{k})} |s(z)|^{2}\ dV(z),\quad w\in X.
	\end{align}
	
	This is proved by the sub-mean value property in \cite[pag. 432 ]{LevJoa}. Indeed, if this is the case then
	\begin{align*}
	\int_{X} |s(w)|^{2}\ d\mu_{k} &\leq Q k^{n} \int_{X} |s(z)|^{2}\left(\int_{X} \chi_{B(w,1/\sqrt{k})}(z)\ d\mu_{k}(w)\right)\ dV(z)\\
	&\lesssim \int_{X} |s(z)|^{2}\ dV(z)
	\end{align*}
	for every $s\in H^{0}(L^{k})$.
	
\end{proof}

\begin{example}[Tautological line bundle]
	We will apply the Theorem~\ref{TeoLogSe} to the complex projective space $\mathbb{CP}^{n}$ with the hyperplane bundle $\mathcal{O}(1)$, endowed with the Fubini-Study metric. Notice that the holomorphic sections to $\mathcal{O}(k)$, the k'th power of $\mathcal{O}(1)$, can be identified with the homogeneous polynomials of degree $k$ in homogeneous coordinates.
	
Using Theorem~\ref{TeoLogSe} we obtain that for a given sequence of measurable subsets $\{\Omega_{k}\}_{k}$ in $\C^{n}$, we have that the following statements are equivalent:
	\begin{enumerate}
		\item[(1)] For all $k\in\N$, it is verified that
		\begin{align*}
		\int_{\Omega_{k}} \frac{|p_{k}(z)|^{2}}{(1+|z|^{2})^{k}}\ \frac{dV(z)}{(1+|z|^{2})^{n+1}}\asymp \int_{\C^{n}} \frac{|p_{k}(z)|^{2}}{(1+|z|^{2})^{k}}\ \frac{dV(z)}{(1+|z|^{2})^{n+1}}
		\end{align*}
		for every polynomial $p_{k}$ of degree less or equal than $k$ in $n$ variables.
		\item[(2)] There is a radius $R$ such that 
		\begin{align*}
	\int_{\Omega_{k}} \frac{\chi_{H(z,R/\sqrt{k})}(w)}{(1+|w|^{2})^{{n+1}}}\ dV(w)\asymp \left(\frac{R}{\sqrt{k}}\right)^{2n}
			\end{align*}
			for all $z\in \C^{n}$ and $k\in\N$, where $$H(z,R/\sqrt{k})=\left\{w\in \C^{n}\ :\ |z-w|<\left|\tan\left(\frac{R}{\sqrt{k}}\right)\right|\cdot |1+z\overline{w}|\right\}.$$
	\end{enumerate}

Notice that the measure $V$ is the Lebesgue measure on $\C^{n}$.
\end{example}

\section*{Acknowledgement}
I would like to express my deep appreciation to Joaquim Ortega Cerdà for his unconditional support and advices to improve in mathematics. In addition, I would like  to show my gratitude to my family and friends of Canary Islands for their encouragement and help to fulfil my dreams.

\bibliographystyle{amsalpha}
\bibliography{biblio}

\providecommand{\bysame}{\leavevmode\hbox to3em{\hrulefill}\thinspace}
\providecommand{\MR}{\relax\ifhmode\unskip\space\fi MR }
\providecommand{\MRhref}[2]{%
  \href{http://www.ams.org/mathscinet-getitem?mr=#1}{#2}
}
\providecommand{\href}[2]{#2}
\begin{thebibliography}{MOC08}

\bibitem[Ber03]{BoB}
B.~Berndtsson, \emph{Bergman kernels related to hermitian line bundles over
  compact complex manifolds}, Explorations in complex and Riemannian geometry,
  Contemp. Math. \textbf{332} (2003), 1--17, \MRNUMB{2016088}\
  \Zbl{1038.32003}.

\bibitem[HJ94]{VHBJ}
V.~Havin and B.~Jöricke, \emph{The uncertainty principle in harmonic
  analysis}, Springer, 1994, ISBN: 978-3-642-78377-7.

\bibitem[Lin01]{LIND}
N.~Lindholm, \emph{Sampling in weighted {$L^{p}$} spaces of entire functions in
  {$\mathbb{C}^{n}$} and estimates of the bergman kernel}, Journal of
  Functional Analysis \textbf{182} (2001), no.~18, 390--426,
  \doi{10.1006/jfan.2000.3733} \ \MRNUMB{1828799} \ \Zbl{1013.32008}.

\bibitem[LOC12]{LevJoa}
N.~Lev and J.~Ortega-Cerdà, \emph{Equidistribution estimates for fekete points
  on complex manifolds}, J. Eur. Math. Soc \textbf{18} (2012), no.~2, 425--464,
  \doi{10.4171/JEMS/594}.

\bibitem[LS74]{LogSer}
V.N. Logvinenko and J.F. Sereda, \emph{Equivalent norms in spaces of entire
  functions of exponential type}, Teor. funktsii, funkt. analiz i ich
  prilozhenia \textbf{20} (1974), 62--78.

\bibitem[Lue81]{Lue1}
D.H. Luecking, \emph{{I}nequalities on {B}ergman {S}paces}, Illinois Journal of
  Mathematics \textbf{25} (1981), no.~1, 1--11, \MRNUMB{602889} \
  \Zbl{0437.30025}.

\bibitem[MOC08]{JM}
J.~Marzo and J.~Ortega-Cerdà, \emph{Equivalent norms for polynomials on the
  sphere}, International Mathematics Research Notices \textbf{2008} (2008),
  \doi{10.1093/imrn/rnm154}.

\end{thebibliography}
\nocite{*}

\end{document}